\documentclass{daj}

\usepackage{amsthm, amsmath, amssymb}

\newtheorem{theorem}{Theorem}[section]
\newtheorem{proposition}[theorem]{Proposition}
\newtheorem{lemma}[theorem]{Lemma}

\newtheorem{corollary}[theorem]{Corollary}
\newtheorem{conjecture}[theorem]{Conjecture}
\newtheorem*{question*}{Question}

\theoremstyle{definition}

\newtheorem{question}[theorem]{Question}

\theoremstyle{remark}

\newcommand{\abs}[1]{\left\lvert#1\right\rvert}

\newcommand{\floor}[1]{\left\lfloor #1 \right\rfloor}

\newcommand{\set}[1]{\left\{ #1 \right\}}

\DeclareMathOperator{\vcdim}{VC-dim}

\newcommand{\EE}{\mathbb{E}}

\newcommand{\PP}{\mathbb{P}}

\newcommand{\ZZ}{\mathbb{Z}}

\dajAUTHORdetails{%
  title = {Efficient
Arithmetic Regularity and Removal Lemmas for Induced Bipartite Patterns}, 
  author = {Noga Alon, Jacob Fox, and Yufei Zhao},
  plaintextauthor = {Noga Alon, Jacob Fox, Yufei Zhao},
  runningtitle = {Efficient Arithmetic Regularity and Removal Lemmas}, 
  keywords = {regularity lemma, removal lemma, property testing, VC dimension, arithmetic regularity, induced patterns},
}   

\dajEDITORdetails{%
   year={2019},
   volume={XX},
   number={3},
   received={13 March 2018},   
   revised={17 March 2019},    
   published={12 April 2019},  
   doi={10.19086/da.7757},       
}   

\begin{document}

\begin{frontmatter}[classification=text]


\author[alon]{Noga Alon\thanks{Supported in part by ISF grant 281/17, GIF grant G-1347-304.6/2016 and the Simons Foundation}}
\author[fox]{Jacob Fox\thanks{Supported by a Packard Fellowship, by NSF Career Award DMS-1352121 and by an Alfred P. Sloan Fellowship}}
\author[zhao]{Yufei Zhao\thanks{Supported by NSF Award DMS-1362326 and DMS-1764176}}

\begin{abstract}
Let $G$ be an abelian group of bounded exponent and $A \subseteq G$. We show that if the collection of translates of $A$ has VC dimension at most $d$, then for every $\epsilon>0$ there is a subgroup $H$ of $G$ of index at most $\epsilon^{-d-o(1)}$ such that one can add or delete at most $\epsilon|G|$ elements to/from $A$ to make it a union of $H$-cosets. 

We also establish a removal lemma with polynomial bounds, with applications to property testing, for induced bipartite patterns in a finite abelian group with bounded exponent. \end{abstract}
\end{frontmatter}


\section{Introduction} \label{sec:intro}

Szemer\'edi's regularity lemma \cite{Sz76} gives a rough structural
decomposition for all graphs and is one of the most powerful tools in
graph theory. A major drawback of the regularity lemma is that the number
of parts in the decomposition grows as an exponential tower of $2$'s
of height a power of $1/\epsilon$, where $\epsilon$ is the regularity
parameter~\cite{Gow97}. A natural question that has been studied by many researchers
is: in what circumstances can one get a more effective bound? Namely,
under what conditions does every graph in a family $\mathcal{F}$ of graphs
necessarily have a partition with much fewer parts, say polynomial in
$1/\epsilon$? One natural condition for a family $\mathcal{F}$ of graphs
is that it is \emph{hereditary}, that is,  if $G \in \mathcal{F}$ then
every induced subgraph of $G$ is also in $\mathcal{F}$. For hereditary
families, it turns out that the bound on the number of parts in a regular
partition is polynomial in $1/\epsilon$ if the neighborhood set system
of every graph in the family has bounded VC dimension, and otherwise
the bound is tower-type. This gives a satisfactory answer to the problem.

A \emph{set system} $\mathcal{S}$ is a collection of subsets of some
ground set $\Omega$. Here we only consider finite $\Omega$. We say
that $U \subseteq \Omega$ is \emph{shattered} by $\mathcal{S}$ if for
every $U' \subseteq U$ there is some $T \in \mathcal{S}$ with $T \cap U =
U'$. The \emph{Vapnik--Chervonenkis dimension} (or \emph{VC dimension})
of $\mathcal{S}$, denoted $\vcdim \mathcal{S}$, is the size of the
largest shattered set.

Let $G$ be a graph. The \emph{neighborhood} $N(v)$ of a vertex $v \in
V(G)$ is the set of vertices adjacent to $v$. The \emph{VC dimension}
of a graph $G$ is defined to be $\vcdim\{N(v):v \in V\}$.

Given a bipartite graph $F$ with vertex bipartition $V(F) = U \cup V$,
we say that a map $\phi \colon V(F) \to V(G)$ \emph{bi-induces} $F$
if for every $(u,v) \in U \times V$, the pair $uv$ is an edge of $F$
if and only if $\phi(u)\phi(v)$ is an edge of $G$. Note that we have no
requirements about edges in $G$ between vertices in the image of $U$,
and likewise with $V$. We say that $G$ contains a \emph{bi-induced copy
of $H$} if there exists a map $\phi$ as above that is injective on each
of $U$ and $V$.\footnote{Having a bi-induced copy of $F$ is weaker than having
an \emph{induced} copy of $F$, where in the latter we also require
that there are no edges in $G$ between vertices in the image of $U$,
and likewise with $V$. Also, an alternative notion of bi-induced copy of $H$ assumes that $\phi$ is injective. The discussed results hold for this 
alternative notion as well.}

It is known that the following are equivalent for a hereditary family
$\mathcal{F}$ of graphs:

\begin{enumerate} 
\item[(1)] The VC dimension of the graphs in $\mathcal{F}$ is uniformly bounded.
\item[(2)] There is a bipartite graph $F$ such that none of the graphs in $\mathcal{F}$ has a bi-induced copy of $F$.
\item[(3)] The family $\mathcal{F}$ has a forbidden induced bipartite graph,
a forbidden induced complement of a bipartite graph, and a
forbidden induced split graph.  
\item[(4)] The number of graphs in
$\mathcal{F}$ on $n$ vertices is at most $2^{n^{2-\epsilon}}$ for some
$\epsilon=\epsilon(\mathcal{F})>0$. In contrast, every other hereditary
family of graphs contains at least $2^{n^2/4}$ labeled graphs on $n$
vertices.
\item[(5)] There is a constant $k = k(\mathcal{F})$ such that every $n$-vertex graph in $\mathcal{F}$ has an equitable vertex partition into at
most $\epsilon^{-k}$ parts such that all but at most an $\epsilon$-fraction
of the pairs of parts have edge density at most $\epsilon$ or at least
$1-\epsilon$. In contrast, every other hereditary family of graphs has
a graph that requires a tower in a power of $1/\epsilon$ parts in any
$\epsilon$-regular equitable vertex partition.  
\end{enumerate}

The above characterizations give an interesting dichotomy between
hereditary families of graphs of bounded VC dimension versus those of unbounded
VC dimension. It shows that families of graphs with bounded VC dimension
have smaller growth and are more structured. The
equivalence of (1) and (4) was given by Alon, Balogh, Bollob\'as, and  Morris~\cite{ABBM}. 
Alon, Fischer, and
Newman~\cite{AFN} proved a bipartite version of the regularity lemma
for graphs of bounded VC dimension, and the version for all graphs is due
to  Lov\'asz and Szegedy \cite{LS}. The proof was simplified with
improved bounds by Fox, Pach, and Suk \cite{FPS}. 
Further results related to the
above equivalences for tournaments can be found in \cite{FGSY}.

A \emph{half-graph} is a bipartite graph on $2k$ vertices
$\{u_1,\ldots,u_k\} \cup \{v_1,\ldots,v_k\}$ such that $u_i$ is adjacent
to $v_j$ if and only if $i \leq j$. Malliaris and Shelah \cite{MaSh}
proved if a graph has no bi-induced copy of the half-graph on $2k$
vertices, then one can partition the vertex set into $\epsilon^{-O_k(1)}$
many parts such that every pair of parts is $\epsilon$-regular (there
are no irregular pairs). Bi-inducing a half-graph is related to a notion
of stability in model theory, and for this reason Malliaris and Shelah
called their result a ``stable regularity lemma''.

The above discussion summarizes some relevant results for graphs. We
now turn our attention to subsets of groups and their associated Cayley
graphs. Let $G$ be a finite abelian group, written additively. Let $A
\subseteq G$. Consider the \emph{Cayley sum graph} formed by taking the
elements of $G$ as vertices, where $x,y \in G$ are adjacent if $x+y
\in A$ (we may end up with some loops; alternatively, we can consider a
bipartite version of this construction). The VC dimension of the graph
corresponds to the VC dimension of the collection of translates of $A$,
which we simply call the \emph{VC dimension of $A$}, defined as
\[
\vcdim A := \vcdim\{ A + x : x \in G\}.
\]

For a bipartite graph $F$ with vertex bipartition $U \cup V$, we say that
a map $\phi\colon V(F) \to G$ \emph{bi-induces $F$ in $A$} if, for every
$(u,v) \in U \times V$, $uv$ is an edge of $F$ if and only if $\phi(u)
+ \phi(v) \in A$. We say that $A$ has a \emph{bi-induced copy of $F$}
if there exists a map $\phi$ as above that is injective on each of $U$
and $V$.

Observe that $A$ has a bi-induced copy of $F$ if its VC dimension is
large enough. To see this, first note that if no pair of vertices in $U$
have identical neighborhoods in $V$, and $A$ has VC dimension at least
$\abs{V}$, then $A$ has a bi-induced copy of $F$. Indeed, we can construct
$\phi$ by mapping $V$ to a subset of $G$ shattered by translates of $A$
(such a choice exists since $\vcdim A \ge \abs{V}$). Since $\phi(V)$ is
shattered, for every $u \in U$, there is some $y_u \in G$ such that $(A
- y) \cap \phi(V) = \phi(N(u))$. Let $\phi$ send $u$ to this $y_u$, for
each $u \in U$. We obtain a map $\phi \colon V(F) \to G$ that bi-induces
$F$, though this map may not be injective on $U$ (it is always injective
on $V$) if some pairs of vertices of $U$ have identical neighborhoods,
but this can be easily fixed\footnote{\label{ft:vc-bi-induce}Consider
the bipartite graph $F_+$ obtained from $F$ by adding $\lceil \log_2
\abs{U} \rceil$ new vertices to the vertex set $V$, and add edges from
the new vertices to $U$ so that no two vertices in $U$ have identical
neighborhoods in $F_+$. By earlier arguments, if $\vcdim A \ge \abs{V} +
\lceil \log_2 \abs{V}\rceil$, then $A$ necessarily contains an bi-induced
copy of $F_+$, and hence a bi-induced copy of $F$.}.

Green \cite{Green05} proved an arithmetic analogue of Szemer\'edi's
regularity lemma for abelian groups. The statement is much simpler in the
case of abelian groups of bounded exponents, which is the main focus of
our paper (some remarks regarding general groups are given in the final
section). For an abelian group $G$ and a subset $A \subseteq G$, a coset
$H+x$ of a subgroup $H$ is called \emph{$\epsilon$-regular} if all the
nontrivial Fourier coefficients of $A \cap (H+x)$, when interpreted
as a subset of $H+x$, are at most $\epsilon$. For each $\epsilon>0$
and positive integer $r$, Green's arithmetic regularity lemma states
that there is $K=K(r,\epsilon)$ such that the following holds. If $G$
has exponent at most $r$ and $A \subseteq G$, then there is a subgroup $H
\subseteq G$ of index at most $K$ such that all but an $\epsilon$-fraction
of the cosets of $H$ are $\epsilon$-regular.

Recently, an arithmetic analog of the Malliaris--Shelah stable regularity
lemma was proved by Terry and Wolf \cite{TeWo} for $G=\mathbb{F}_p^n$
with $p$ fixed. It was shown that if $A \subseteq G$ has no bi-induced copy
of a half-graph on $2k$ vertices, then there is a subgroup $H$ of $G$ of
index at most $e^{\epsilon^{-O_{k,p}(1)}}$ such that for every $x \in G$,
one has either $\abs{A\cap (H+x)} \le \epsilon \abs{H}$ or $\abs{A\cap
(H+x)} \ge (1-\epsilon) \abs{H}$. Here the subscripts on the $O_{k,p}(1)$
mean that the constant is allowed to depend on $k$ and $p$. The result
was subsequently extended to general groups by Conant, Pillay, and Terry
\cite{CPT}, who showed that for every finite group $G$, if $A \subseteq G$
has no bi-induced copy of the half-graph on $2k$ vertices, then there
is a normal subgroup $H$ of $G$ of index $O_{k,\epsilon}(1)$ such that
there is some union $S$ of $H$-cosets such that $\abs{A \Delta S} \le
\epsilon \abs{H}$, where $A \Delta B = (A \setminus B) \cup (B \setminus
A)$ denotes the symmetric difference. However, the general group version
of the theorem~\cite{CPT} gives no quantitative bounds on the index of
$H$ due to the model theoretic tools involved in its proof.

We saw earlier that forbidding a fixed bi-induced bipartite graph
implies bounded VC dimension. Our first main result generalizes a
variant of Terry and Wolf's result to sets of bounded VC dimension,
and gives bounds of polynomial order in $1/\epsilon$. Its proof can be
found in Section~\ref{sec:reg}.

\begin{theorem}[Regularity lemma] \label{thm:reg}
Fix positive integers $r$ and $d$. If $G$ is a finite abelian group with
exponent at most $r$, and $A \subseteq G$ has VC dimension at most $d$,
then for every $\epsilon > 0$ there is a subgroup $H$ of $G$ of index at
most $\epsilon^{-d -o(1)}$ such that $\abs{A\Delta S} \le \epsilon\abs{G}$
for some $S \subseteq G$ which is a union of cosets of $H$.

Here $o(1)$ is some quantity that goes to zero as $\epsilon \to 0$,
at a rate possibly depending on $r$ and $d$.
\end{theorem} 

We also prove a removal lemma for bi-induced copies of a fixed bipartite
graph. Let us first recall the classical graph removal lemma. We say
that an $n$-vertex graph is \emph{$\epsilon$-far} from some property if
one needs to add or delete more than $\epsilon n^2$ edges to satisfy the
property. The triangle removal lemma\footnote{The removal lemma is often
stated in the contrapositive, which better explains the name ``removal
lemma'': if triangle density of a graph is at most $\delta(\epsilon)>0$,
then the graph can be made triangle-free by deleting $\epsilon n^2$ edges}
says that if an $n$-vertex graph is $\epsilon$-far from triangle-free,
then its triangle density is at least $\delta(\epsilon)>0$. The original
graph regularity proof \cite{RS} of the triangle removal lemma shows that we may take
$1/\delta(\epsilon)$ to be a tower of 2's of height $\epsilon^{-O(1)}$, which
was  improved to height 
$O(\log (1/\epsilon))$ in \cite{Fox11}. It is known that
there exists a constant $c>0$ such that the bound in the triangle removal
lemma cannot be improved to $\delta = \epsilon^{-c \log (1/\epsilon)}$
(see \cite{CF13} for a survey on graph removal lemmas). There is also a
removal lemma for induced subgraphs \cite{AFKS}, initially proved using
a so-called \emph{strong regularity lemma}, though better bounds were
later obtained in \cite{CF12}.

An arithmetic analog of the graph removal lemma was first proved by Green~\cite{Green05} for ``complexity 1'' patterns such as $x+y+z=0$ using his arithmetic regularity lemma. Kr\'al', Serra, and Vena~\cite{KSV09} later showed that Green's arithmetic removal lemma can be deduced as a consequence of the graph removal lemma. More general arithmetic removal lemmas for linear systems were later proved as a consequence of the hypergraph removal lemma~\cite{KSV12,Sha10}. We refer to the references for precise statements. Note that the reduction from the arithmetic removal lemma to the (hyper)graph removal lemma fails for induced patterns. It remains open to find a general induced arithmetic removal lemmas \cite[Conjecture 5.3]{Sha10}.

Our second main result gives an arithmetic analog of the removal lemma,
with polynomial bounds, for bi-induced patterns.  We say that $A\subset
G$ is \emph{$\epsilon$-far from bi-induced-$F$-free} if $A'\subset
G$ contains a bi-induced copy of $F$ whenever $\abs{A \Delta A'} \le
\epsilon \abs{G}$. Here is our second main result, whose proof can be
found in Section~\ref{sec:removal}.

\begin{theorem}[Removal lemma] 
\label{thm:removal}
Fix a positive integer $r$ and a bipartite graph $F$. Let $G$ be a finite
abelian group with exponent at most $r$. For every $0 < \epsilon <1/2$,
if $A\subseteq G$ is $\epsilon$-far from bi-induced-$F$-free, then the
probability that a uniform random map $\phi \colon V(F) \to G$ bi-induces
$F$ is at least $\epsilon^{O(\abs{V(F)}^3)}$.
\end{theorem}

We mention an application to property testing. The removal lemma gives
a polynomial-time randomized sampling algorithm for distinguishing
sets $A \subseteq G$ that are bi-induced-$F$-free from those that are
$\epsilon$-far from bi-induced-$F$-free. Indeed, sample a random map
$\phi \colon V(F) \to G$, and output YES if $\phi$ bi-induces $F$ and
is injective on each vertex part of $F$, and otherwise output NO. If
$A$ is bi-induced-$F$-free, then the algorithm always outputs NO. On
the other hand, if $F$ is $\epsilon$-far from bi-induced-$F$-free,
then by the theorem above, the algorithm outputs YES with probability
at least $\epsilon^{O_F(1)}$, provided that $G$ is large enough,
so that $\phi$ is injective with high probability. We can then repeat
the experiment $\epsilon^{-O_F(1)}$ times to obtain a randomized
algorithm that succeeds with high probability.

\section{Regularity lemma} \label{sec:reg}

In this section, we prove Theorem~\ref{thm:reg}.

We say that a set system $\mathcal{S}$ on a finite ground set $\Omega$ is
\emph{$\delta$-separated} if $\abs{S \Delta T} \ge \delta \abs{\Omega}$
for all distinct $S,T \in \mathcal{S}$. We quote a bound on the size of
a $\delta$-separated system.

\begin{lemma}[Haussler's packing lemma~\cite{H95}] 
\label{lem:haussler}
Let $d,\delta >0$. If $\mathcal{S}$ is a $\delta$-separated set system
of VC dimension at most $d$, then $\abs{\mathcal{S}} \le (30/\delta)^d$.
\end{lemma}

By taking a maximal $\delta$-separated collection of translates of $A \subseteq G$, we deduce, below, that $A$ must be $\delta$-close to many of its own translates.

\begin{lemma} \label{lem:invariant}
Let $G$ be a finite abelian group, and $A \subseteq G$ a subset with VC dimension at most $d$, and $\delta > 0$. Then
\[
\abs{\{x : \abs{A \Delta (A + x)} \le \delta \abs{G}\}} \ge (\delta/30)^d \abs{G}.
\]
\end{lemma}

\begin{proof}
Let $W$ be a maximal subset of $G$ such that $\abs{(A + w) \Delta (A + w')} > \delta \abs{G}$ for all distinct $w,w' \in W$. We have $\abs{W} \le (30/\delta)^d$ by Lemma~\ref{lem:haussler}. Let
\[
B = \{x \in G : \abs{A \Delta (A+x)} \le \delta \abs{G}\}.
\]
Since $W$ is maximal, for every $x \in G$, there is some $w \in W$
such that $\abs{(A+x)\Delta (A+w)} \le \delta \abs{G}$, which implies
$x - w \in B$. Hence $G = \bigcup_{w \in W}(B+w)$. Therefore $\abs{B} \ge
\abs{G}/\abs{W} \ge (\delta/30)^d\abs{G}$.
\end{proof}

We quote a result from additive combinatorics. We use the following standard notation: $A+A = \{a+b : a,b\in A\}$, $A - A = \{a-b : a,b\in A\}$, and $k A = A + \cdots + A$ ($k$ times).

\begin{theorem}[Bogolyubov--Ruzsa lemma for groups with bounded exponent] \label{thm:bog-exp}
	Let $G$ be an abelian group of exponent at most $r$, and $A \subseteq G$ a finite subset with $\abs{A+A}\le K\abs{A}$. Then $2A-2A$ contains a subgroup of $G$ of size at least $c_r(K)\abs{A}$ for some constant $c_r(K) > 0$.
\end{theorem}

The name ``Bogolyubov--Ruzsa lemma'' was given by Sanders~\cite{San12},
who proved the theorem with the current best bound $c_r(K) =
e^{-O_r(\log^4 2K)}$ (see \cite[Theorem 11.1]{San12}). We refer the
readers to the introductions of \cite{San12,San13} for the history of
this result. A version of the theorem for $G = \ZZ$ was initially proved
by Ruzsa~\cite{Ruz94} as a key step towards his proof of Freiman's
theorem. The assertion of the 
polynomial Freiman--Ruzsa conjecture, a central open problem
in additive combinatorics, would follow from an improvement of the bound
to $c_r(K) = K^{-O_r(1)}$.

In our next lemma, we start from the conclusion of
Lemma~\ref{lem:invariant}, which gives us a large set $B$ such that $A
\approx A+x$ for all $x \in B$. Consider the sequence $B, 2B, 4B, 8B,
\dots$. Since $B$ is large, the size of $2^iB$ cannot keep on growing,
so we can find a set $B' = 2^iB$ with small doubling $\abs{B'+B'}
\le K\abs{B'}$, and $i$ not too large. Theorem~\ref{thm:bog-exp} then
implies that $2B'-2B'$ contains a large subgroup, in which every element $x$
satisfies $A \approx A +x$, which is close to what we need.
 
\begin{lemma} \label{lem:reg-key}

	Fix a positive integer $r$. Let $G$ be a finite abelian group of
	exponent at most $r$. Let $0 < \delta < 1/2$, $C > 0$, and $A \subseteq G$. Let
	$B = \{x \in G : \abs{A\Delta(A+x)}\le \delta \abs{G}\}$. Suppose $\abs{B} \ge \delta^C \abs{G}$. Then there exists a subgroup $H$
	of $G$ with $\abs{H} \ge \delta^{o(1)} |B|$ such that $\abs{A
	\Delta (A+x)} \le \delta^{1-o(1)} \abs{G}$ for all $x \in H$,
	and furthermore there exists a union $S$ of $H$-cosets such that
	$\abs{A \Delta S} \le \delta^{1-o(1)} \abs{G}$.
	Here $o(1)$ is a quantity that goes to zero as $\delta \to 0$, at a rate that may depend on $r$ and $C$.
\end{lemma}

\begin{proof}
	Let $K = K(\delta) > 1$ to be decided. We cannot have $|2^{i+1} B| > K |2^{i} B|$ for every $0 \le i \le \log_K(|G|/|B|)$ since otherwise we would have $|2^i B| > |G|$ for some $i$, which is impossible as $2^i B$ is a subset of $G$. Thus $|2^{i+1} B| \le K |2^{i} B|$ for some $i \le \log_K(|G|/|B|) \le C \log_K (1/\delta)$, and letting $\ell = 2^i$, we have 
	\begin{equation}
		\label{eq:2l-K}
	|2\ell B| \le K|\ell B| \quad \text{with} \quad \ell \le 2^{C \log (1/\delta)/\log K} = \delta^{- O(1/\log K)}.
	\end{equation}
	
	Since $|(A + x) \Delta A | \le \delta |G|$ for all $x \in B$, we have, by the triangle inequality,
	\[
	|(A + x + y) \Delta A| \le |(A + x + y) \Delta (A + y)| + |(A + y) \Delta A| = |(A + x) \Delta A| + |(A + y) \Delta A| \text{ for all } x,y \in G.
	\]
	Thus
	\begin{equation}
		\label{eq:2l-2l}
	|A \Delta (A + x) | \le 4\ell \delta \abs{G} \quad \text{ for all } x \in 2\ell B - 2\ell B.
	\end{equation}
	By Theorem~\ref{thm:bog-exp} and \eqref{eq:2l-K}, $2\ell B - 2\ell B$ contains a subgroup $H$ of $G$ with $\abs{H} \ge c_r(K) \abs{\ell B} \ge c_r(K) \abs{B}$. This would complete the proof of the first claim in the lemma provided that $K = K(\delta) \to \infty$ slowly enough as $\delta \to 0$ so that $c_r(K) = \delta^{o(1)}$ (then $\ell \le \delta^{-O(1/\log K)} = \delta^{-o(1)}$). Concretely, Theorem~\ref{thm:bog-exp} with Sander's $c_r(K) = e^{-O_r(\log^4 2K)}$ allows us to take $K(\delta) = \exp((\log 1/\delta)^{1/5})$, say, so that all the $o(1)$'s in the exponents decay as $(\log(1/\delta))^{-1/5}$.
	
	For the second claim, let $S$ be the union of all $H$-cosets $y+H$ with $\abs{A \cap (y+H) } \ge \abs{H}/2$. 
	Then
	\begin{align*}
	\abs{A \Delta S}
	&= \sum_{y \in G/H}  \min\set{\abs{A \cap (y+H)}, \abs{H} - \abs{A \cap (x+H)}}
	\\
	&\le \sum_{y \in G/H}  \frac{2}{\abs{H}}\abs{A \cap (y+H)}( \abs{H} - \abs{A \cap (y+H)})
	\\
	&= \frac{1}{\abs{H}} \sum_{x \in H} \abs{A \Delta (A+x)} \quad \text{\footnotesize[counting pairs in $A \times (G\setminus A)$ lying in the same $H$-coset]}
	\\
	&\le 4\ell \delta \abs{G} = \delta^{1-o(1)} \abs{G}. \quad \text{\footnotesize[by \eqref{eq:2l-2l}]}
	\end{align*}
\end{proof}

The regularity lemma, Theorem~\ref{thm:reg}, then follows immediately after combining Lemmas~\ref{lem:invariant} and \ref{lem:reg-key}.

\medskip 

Instead of applying the Bogolyubov--Ruzsa lemma as we do above, it is
also possible to prove Lemma~\ref{lem:reg-key} using Freiman's theorem
for groups of bounded exponent:

\begin{theorem} [Ruzsa~\cite{Ruz99}] \label{thm:Freiman-groups} 
If $A$ is
a finite subset of an abelian group of exponent at most $r$ such that
$\abs{A+A} \le K\abs{A}$, then $A$ is contained in a subgroup of size
$O_{r, K}(1)\abs{A}$.
\end{theorem}

At the point in the proof of Lemma~\ref{lem:reg-key}
where we apply Theorem~\ref{thm:bog-exp}, we can instead apply Theorem~\ref{thm:Freiman-groups} to contain $\ell
B$ inside a subgroup of size $\delta^{-o(1)}\abs{\ell B}$. Now we apply a corollary of Kneser's theorem.

\begin{theorem}[Kneser's theorem~\cite{Kn}; see~{\cite[Theorem 5.5]{TV}}] 
Let $G$ be an abelian group and $A, B$ finite non-empty subsets. If $|A| + |B| \leq |G|$ then there
is a finite subgroup $H$ of $G$ such that 
\[
\abs{A + B} \ge  \abs{A + H} + \abs{B + H} - \abs{H}
 \ge \abs{A} + \abs{B} - \abs{H}.
\]
The subgroup $H$ can be taken to be the stabilizer of $A+B$: 
\[
H = \{ g \in G : g + ( A + B ) = ( A + B ) \}.
\]
\end{theorem}

\begin{corollary}
If $G$ is an abelian group, $t$ is a positive integer, and $A \subset G$
has $|A| \geq |G|/t$ and $A$ generates $G$, then $2t A = G$.
\end{corollary} 

\begin{proof}
For any $i$ such that $(i+1) A \ne G$, applying Kneser's theorem to the sets $iA$ and $A$ gives us a subgroup $H$ so that $|(i+1) A| \ge |iA+H| + |A+H| - |H| \ge |iA| + |A|/2$ (since $A$ generates $G$, $A+H$ is a union of at least two cosets of $H$, so $|H| \le |A+H|/2$ and $|A+H| \geq |A|$). Iterating gives $|2tA| \ge t |A| \ge |G|$.
\end{proof}

Let us continue with our discussion of the alternative approach to proving Lemma~\ref{lem:reg-key}. Since $\ell B$ occupies a $\delta^{o(1)}$-fraction of some subgroup, by the above corollary, $\ell' B$ is a subgroup (playing the role of $H$ in the first proof) for some $\ell' = \delta^{-o(1)}\ell$. From this
point we can proceed as the rest of the proof of Lemma~\ref{lem:reg-key}.

\section{A strengthened regularity lemma}

In the next section, we prove a removal lemma for bi-induced patterns. The regularity lemma we stated in Theorem~\ref{thm:reg} seems
not quite strong enough to establish the removal lemma. Below we prove a
strengthening, where the VC dimension hypothesis is weakened to a more
robust one. Instead of requiring that $A$ has bounded VC dimension,
we will ask that, with probability at least 0.9, say, the VC dimension
of the collection of translates of $A$ is bounded if we restrict the
ground set $G$ to a random set. We state the result below in the form
of two alternatives: either $A$ has high VC dimension when sampled,
or it satisfies a regularity lemma with polynomial bounds.

\begin{proposition}[Regularity lemma with robust VC dimension hypothesis] 
\label{prop:robust-reg}
	Fix positive integers $r$ and $d$.
	Let $G$ be a finite abelian group of exponent at most $r$. 
Let $A \subseteq G$. One of the following must be true for every
small $\epsilon>0$:
	\begin{enumerate}
	\item[(a)] For some $k = \epsilon^{-d-o(1)}$, if $X$  and $Y$ are random $k$-element subset of $G$, then we have $\vcdim\{(A+ x) \cap Y : x \in X\} > d$ with probability at least $0.9$.
	\item[(b)] There exists a subgroup $H$ of $G$ of index at most $\epsilon^{-d-o(1)}$ such that $\abs{A \Delta S} \le \epsilon\abs{G}$ for some union $S$ of $H$-cosets.
	\end{enumerate}
	Here $o(1)$ refers to a quantity that goes to zero as $\epsilon \to 0$, at a rate that can depend on $r$ and $d$.
\end{proposition}

Recall that Lemma~\ref{lem:invariant} tells us that if $\vcdim A \le
d$, then  $B = \{x \colon \abs{A \Delta (A+x)} \le \delta \abs{G}\}$
has size at least $(\delta/30)^d\abs{G}$. We will derive a similar
bound for $B$ under the weaker hypothesis, namely the negation of
(a), from which we can deduce (b) using
Lemma~\ref{lem:reg-key} as in the proof of the previous regularity lemma
Theorem~\ref{thm:reg}.

\begin{lemma} \label{lem:ind-set}
Let $k \le n/2$ be positive integers.
In an $n$-vertex graph with maximum degree at most $n/k$, a random $k$-element subset of the vertices contains an independent set of size at least $k/4$ with probability at least $1 - e^{-k/8}$.
\end{lemma}

\begin{proof}
Let $v_1, \dots, v_k$ be a sequence of $k$ vertices chosen uniformly
at random without replacement. Let $I$ be the independent set formed
greedily by, starting with the empty set, putting each $v_i$, sequentially
as $i=1, 2, \dots$, into $I$ if doing so keeps $I$ an independent
set. During the process, when at most $k/4$ elements are added to
$I$, the probability that a new $v_i$ is added to $I$ is at least $1 -
\frac{(k/4)(n/k)}{n-k} \ge \frac{1}{2}$, since among the remaining $n-k$
vertices, at most $(k/4)(n/k)$ of them are adjacent to vertices already
added to $I$ at this point. It follows that $|I|$ stochastically dominates
$\min\{X, k/4\}$, where $X$ is distributed as $\operatorname{Binomial}(k,
1/2)$. Thus $\PP(|I| < k/4) \le \PP(X < k/4) \le e^{-k/8}$ by the Chernoff
bound. Therefore, $\{v_1, \dots, v_k\}$ contains an independent
set $I$ of size at least $k/4$ with probability at least $1 - e^{-k/8}$.
\end{proof}

We recall a basic result on VC dimension.

\begin{theorem}[Sauer--Perles--Shelah theorem~\cite{Sau,She,VC}] \label{thm:vc}
	If $\mathcal{S}$ is a set system on a ground set of $n$ elements 
with VC dimension at most $d$, 
then $\abs{\mathcal{S}} \le \sum_{i=0}^d \binom{n}{i} \le 2n^d$.
\end{theorem}

\begin{lemma} \label{lem:sep-sample}
	Let $0 < \delta < 1$, and let $m$ and $d$ be positive integers.
	Let $\mathcal{S}$ be a $\delta$-separated set system. Suppose that for a uniformly random $m$-element subset $M$, the restricted set system $\mathcal{S}|_M := \{T \cap M : T \in \mathcal{S}\}$ has VC dimension at most $d$ with probability at least $3m^{2d}(1-\delta)^m$. Then $\abs{\mathcal{S}} \le 2m^d$.
\end{lemma}

\begin{proof}
	Assume for contradiction that there exists such a set system with $\abs{\mathcal{S}} = 2m^d + 1$.
	Let $n$ be the size of the ground set.
	We have $\abs{S \Delta T} \ge \delta n$ for all distinct $S,T \in \mathcal{S}$. Then, for each pair of distinct $S,T \in \mathcal{S}$, with probability at least $1 - (1-\delta)^m$, $M$ intersects $S \Delta T$, so that $S$ and $T$ remain distinct when restricted to $M$. Taking a union bound over all pairs of sets in $\mathcal{S}$, we see that with probability at least $1 - \binom{\abs{\mathcal{S}}}{2}(1-\delta)^m \ge 1 - 3m^{2d}(1-\delta)^m$, all sets in $\mathcal{S}$ remain distinct when restricted to $M$, in which case $\vcdim(\mathcal{S}|_M) > d$ by Theorem~\ref{thm:vc} as $\abs{\mathcal{S}} > 2m^d$, a contradiction to the hypothesis.
\end{proof}

\begin{lemma} \label{lem:robust-vc-to-ball}
Let $m$ and $d$ be positive integers and $0 < \delta < 1$. Let $G$ be a finite abelian group of order at least $24m^d$. Let $X$ be a random $12m^d$-element subset of $G$, and $Y$ a random $m$-element subset of $G$. If $\vcdim\{(A+x) \cap Y : x \in X\} \le d$ with probability at least $e^{-m^d} + 3m^{2d}(1-\delta)^m$, then $B = \{x : \abs{A \Delta(A+x)} \le \delta \abs{G}\}$ has at least $\abs{G}/(12m^d)$ elements.
\end{lemma}

\begin{proof}
	Suppose, on the contrary, that $\abs{B} <
	\abs{G}/(12m^d)$. Consider the Cayley graph on $G$ generated
	by $B \setminus \{0\}$, i.e., there is edge between $x,y \in G$
	whenenver $x-y \in B$. Applying Lemma~\ref{lem:ind-set} with $k
	= 12m^d$ to this graph, we find that with probability at least
	$1-e^{-m^d}$, a random $12m^d$-element subset $X \subseteq G$
	contains an independent set $I\subseteq X$ with $\abs{I} \ge 3m^d$
	with respect to this graph, i.e., $\abs{(A + x) \Delta (A+y)}
	> \delta \abs{G}$ for all distinct $x,y \in I$. It follows,
	by union bound and averaging, that we can fix such a set $X$ so
	that $\vcdim\{(A+x)\cap Y : x \in X\} \le d$ with probability
	at least $3m^{2d}(1-\delta)^m$ for the random $m$-element set
	$Y \subseteq G$.

	Note that $\{A + x : x \in I\}$ is a $\delta$-separated set system
	with ground set $G$. Furthermore, $\vcdim\{(A+x)\cap Y : x \in I\}
	\le \vcdim\{(A+x)\cap Y : x \in X\} \le d$ with probability at
	least $3m^{2d}(1-\delta)^m$. So by Lemma~\ref{lem:sep-sample},
	we have $\abs{I} \le 2m^d$, which contradicts the bound $\abs{I}
	\ge 3m^d$ above.
\end{proof}

\begin{proof}[Proof of Proposition~\ref{prop:robust-reg}]
	Let $0 < \delta <1/2$. Consider $B = \{x : \abs{A \Delta (A+x)} \le \delta \abs{G}\}$. Choose $m = C \delta^{-1}\log(1/\delta)$ where $C$ is a sufficiently large constant. Then $e^{-m^d} < 1/20$ and $3m^{2d}(1-\delta)^m < 2m^{2d} e^{-\delta m} < 1/20$.
	
	If $\abs{B} < \abs{G}/(12m^d)$, then by Lemma~\ref{lem:robust-vc-to-ball}, if $X$ and $Y$ are random $2m^d$-element subsets of $G$, then $\vcdim\{(A+x) \cap Y : x \in X\} > d$ with probability at least $0.9$.
		
	On the other hand, if $\abs{B} \ge \abs{G}/(12m^d)$, then by Lemma~\ref{lem:reg-key} there exists a subgroup $H$ of $G$ with $\abs{H} \ge \delta^{o(1)} |B| \ge \delta^{d + o(1)} |H|$ such that $\abs{A \Delta S} \le \delta^{1-o(1)} \abs{G}$ for some union $S$ of $H$-cosets. By choosing $\delta = \epsilon^{1+o(1)}$ so that $\abs{A \Delta S} \le \epsilon \abs{G}$, we obtain the desired result.
\end{proof}

\section{Removal lemma} \label{sec:removal}

In this section, we prove the removal lemma, Theorem~\ref{thm:removal},
for bi-induced patterns.

The result is analogous to the induced removal lemma \cite{AFKS} which 
can be proved using a strong version of the graph regularity lemma. The
usual way of proving the strong graph regularity lemma involves iteratively applying the
graph regularity lemma. For our arithmetic setting, as we are concerned with bi-induced patterns, the situation is a
bit easier: we simply apply the regularity lemma, Proposition~\ref{prop:robust-reg}, twice, where the second
time we choose a smaller error parameter compared to the first time.
If option (a) holds
either time, then we can extract a bi-induced copy of $F$ from each sample
with high VC dimension. Otherwise, (b) holds, and we can modify $A$ by
a small amount to $A'$, which must also have a bi-induced copy of $F$
(since $A$ is $\epsilon$-far from bi-induced-$F$-free). The set $A'$ is a
union of $H$-cosets where $H$ is a subgroup of bounded index, and we will
show that a single bi-induced copy of $F$ in $A'$ leads to many copies.

\begin{proof}[Proof of Theorem \ref{thm:removal}]
	Let $V(F) = U \cup V$ be the vertex bipartition of $F$, where
	$\abs{U} \ge \abs{V}$. Let $d =\abs{U} + \lceil \log_2\abs{U}
	\rceil$.

	We may assume that $\abs{G} \ge \epsilon^{-\Omega(\abs{V(F)}^2)}$ or
	else the conclusion is automatic from just a single bi-induced
	copy of $F$ in $A$.

Suppose, for some $k = \epsilon^{-O(\abs{V(F)})}$, with probability
at least 0.9, random $k$-element subsets $X, Y \subseteq G$ satisfy
$\vcdim\{(A + x) \cap Y : x \in X\} > d$, in which case there exist
injective maps $U \to X$ and $V \to Y$ that bi-induce $F$ in $A$
by footnote~\ref{ft:vc-bi-induce}. Then the probability that random
injections $U \to G$ and $V \to G$ bi-induce $F$ is at least $0.9
\binom{k}{|U|}^{-1}\binom{k}{|V|}^{-1} \ge 0.9 k^{-|U|-|V|} \ge
\epsilon^{O(\abs{V(F)}^2)}$, since we can choose the random injection
$U \to G$ by first choosing the random $k$-element subset $X \subset
G$ and then taking a random injection $U \to X$, and similarly with
$V$. With probability $1-O_F(\abs{G}^{-1})$ a random map $V(F) \to G$
is injective on $U$ and $V$, so it bi-induces $F$ with probability at
least $\epsilon^{O(\abs{V(F)}^2)}$.

	We apply Proposition~\ref{prop:robust-reg} with two different
	parameters $\epsilon_1 = \epsilon/10$ and some $\epsilon_2$ to be specified later.
	If option (a) is true in either
	case, then the previous paragraph implies the conclusion of
	the Theorem. Otherwise, we obtain subgroups $H_1$ and $H_2$
	of $G$, such that for each $i \in \{1,2\}$, one has $h_i :=
	\abs{G}/\abs{H_i} \le \epsilon_i^{-d-o(1)}$ and there exists some
	union $S_i$ of $H_i$-cosets satisfying $\abs{A \Delta S_i} \le
	\epsilon_i\abs{G}$. Furthermore, we choose $\epsilon_2$ so that $h_1	\epsilon_2 \abs{U}\abs{V} = 1/8$. In particular, $\epsilon_2 \ge  \epsilon^{d + o(1)}$.

	Let $H = H_1 \cap H_2$. So $|G|/|H| \le h_1h_2 \le
	\epsilon^{-d^2-d-o(1)}$. We say that a coset $x + H$ of $H$ is
	\emph{good} if $\abs{A \Delta (x+H)}/\abs{H}$ is within $\eta
	:= 1/(2|U||V|)$ of $0$ or $1$, and \emph{bad} otherwise. At
	most an $\epsilon_2/\eta$-fraction of $H$-cosets are bad, since
	otherwise bad $H$-cosets would together contribute more than
	$(\epsilon_2/\eta) \eta \abs{G}$ elements to $A \Delta S_2$
	as $S_2$ is also a union of $H$-cosets, but this is impossible
	as $\abs{A \Delta S_2}\le \epsilon_2\abs{G}$.

	Pick an arbitrary subgroup $K$ of $G$ containing exactly one
	element from each coset of $H_1$ (so that $G = H_1 \oplus
	K$ as a direct sum). Let $z \in H_1$ be chosen uniformly at
	random. Then $z + K + H$ is a union of $\abs{K} = h_1$ many
	$H$-cosets. For each $y \in K$, the random $H$-coset $z+ y +
	H$ is uniformly chosen from all $H$-cosets in $y+H_1$. Applying
	the union bound, we see that the probability that $z + K +
	H$ contains a bad $H$-coset is at most $h_1 \epsilon_2/\eta <
	2h_1\epsilon_2 \abs{U}\abs{V} < 1/2$.

	Let $A' \subseteq G$ be the union of $H_1$-cosets $y + H_1$,
	ranging over all $y \in K$ with $\abs{A \cap (z + y +
	H)} \ge \abs{H}/2$. Since $A'$ and $S_1$ are both unions
	of $H_1$-cosets, we can apply linearity of expectation over
	$H_1$-cosets to deduce that $\EE[\abs{A' \Delta S_1}] \le 2\abs{A
	\Delta S_1} \le 2\epsilon_1\abs{G}$, and hence $\EE[\abs{A'
	\Delta A}] \le \EE[\abs{A' \Delta S}] + \abs{A \Delta S} \le
	3\epsilon_1\abs{G}$. Thus, with probability at least $1/2$,
	one has $\abs{A' \Delta A}/\abs{G} \le 6\epsilon_1 < \epsilon$.

	Therefore there is some instance such that $\abs{A' \Delta A} <
	\epsilon \abs{G}$, and $z + K + H$ is a union of good $H$-cosets.

	Since $A$ is $\epsilon$-far from bi-induced-$F$-free, $A'$
	contains a bi-induced-copy of $F$. So there exist $x'_u, y'_v
	\in G$ over $u \in U$ and $v \in V$ such that for all $(u,v)\in
	U \times V$, one has $x'_u + y'_v \in A'$ if and only if $uv\in
	E(F)$. Since $A'$ is a union of $H_1$-cosets, and there is an
	element of $K$ in every $H_1$-coset, we may assume that $x'_u
	\in K$ for each $u \in U$ and $y'_v \in z + K$ for each $v \in V$.

	Consider independent and uniform random elements $x_u \in x'_u +
	H$ for each $u \in U$, and $y_v \in y'_v + H$ for each $v \in
	V$. For each $(u,v) \in U\times V$, the random element $x_u +
	y_v$ is distributed uniformly in the $H$-coset $x'_u + y'_v + H$,
	which is a good $H$-coset since $x'_u + y'_v \in z + K$ as $K$
	is a subgroup. So with probability at least $1-\eta$, one has
	$x_u+y_v \in A$ if and only if $x'_u + y'_v \in A'$, which in
	turn occurs if and only if $uv \in E(F)$. Taking a union bound
	over $(u,v)\in U \times V$, the following holds with probability
	at least $1 - \abs{U}\abs{V} \eta= 1/2$: for every $(u,v)\in
	U\times V$, one has $x'_u + y'_v \in A$ if and only if $uv \in
	E(F)$. Since each $x_u$ and $y_v$ is restricted to a single
	$H$-coset, it follows that a uniform random map $\phi \colon
	V(F) \to G$ bi-induces $F$ with probability at least $\frac12
	(\abs{H}/\abs{G})^{\abs{V(F)}} \ge
\epsilon^{(d^2+d+o(1))\abs{V(F)}}$.
\end{proof}

\section{Concluding remarks}

We conjecture that the result can be extended to general groups, not necessarily abelian. 

\begin{conjecture}
Fix positive integers $r$ and $d$. Let $G$ be a group of exponent at most
$r$, and $A \subseteq G$ a subset with VC dimension at most $d$. Then, for
every $\epsilon > 0$, there is a normal subgroup $H$ of $G$ of index at
most $\epsilon^{-O_{r,d}(1)}$ so that $\abs{A \Delta S} \le \epsilon \abs{G}$
for some union $S$ of $H$-cosets.
\end{conjecture} 

A special case of the conjecture, though with a somewhat stronger but
non-quantitative conclusion, where one forbids a half-graph of fixed size
(instead of assuming bounded VC dimension), was recently established by
Conant, Pillay, and Terry~\cite{CPT} using model theoretic tools. 

Note that the bounded exponent hypothesis in the conjecture above cannot be dropped. Indeed, if $G = \ZZ/p\ZZ$ with $p$ prime, and $A = \{1, 2, \dots, \floor p/2 \rfloor\}$, then $\vcdim A \le 3$, while $G$ has no non-trivial subgroups, so the conclusion of the conjecture is false. Nonetheless, there may be regularity lemmas using other structures in addition to subgroups. An example of such a result is discussed later in this section.

We also conjecture that the removal lemma should generalize to arbitrary groups as well, although it seems to be open even for the general abelian groups.

\begin{conjecture}
	Fix a bipartite graph $F$. Let $G$ be a finite group. For every $0 < \epsilon < 1/2$, if $A \subseteq G$ is $\epsilon$-far from bi-induced-$F$-free, then the probability that a uniform random map $\phi \colon V(F) \to G$ bi-induces $F$ is at least $\epsilon^{O_F(1)}$.
\end{conjecture}

It seems likely that the theory developed by Breuillard, Green, and Tao
\cite{BGT1,BGT2} on the structure of approximate groups should be useful
in the case of nonabelian groups. We hope to study these problems in the future.

\medskip

In classical results in additive combinatorics, such as Freiman's
theorem, when the ambient group does not have many subgroups, generalized
progressions and Bohr sets often play the role of subgroups when the group
does not have many subgroups. For example, in Green and Ruzsa's~\cite{GR}
extension of Freiman's theorem to general abelian groups, the basic
structural objects are \emph{coset progressions}, which are sets of the
form $P = Q + H$, where $H$ is a subgroup, and $Q$ is some generalized
arithmetic progression $\{x_0 + i_1 x_1 + \cdots + i_d x_d : 0 \le i_j
< \ell_j \text{ for each } j\}$, and the sum $Q+H$ is a direct sum in
the sense that every element in $Q+H$ can be written as $q+h$ with
$q\in Q$ and $h \in H$ in a unique way. We say that the progression
is \emph{proper} if all the terms $x_0 + i_1 x_1 + \cdots + i_d x_d$ in $Q$
are distinct. We call $d$ the \emph{dimension} of the progression.

The Bogolyubov--Ruzsa lemma, Theorem~\ref{thm:bog-exp}, holds for general abelian groups (see~\cite[Section 5]{GR}; also see \cite{San12}).

\begin{theorem}[Bogolyubov--Ruzsa lemma for general abelian groups] \label{thm:bog}
	Let $G$ be an abelian group, and $A \subseteq G$ a finite set such that $|A + A| \le K|A|$. Then $2A - 2A$ contains a proper coset progression $P$ of dimension at most $d(K)$ and size at least $c(K)|A|$, for some constants $c(K), d(K) > 0$.
\end{theorem}

By modifying the proof of Theorem~\ref{thm:bog-exp} so that we apply
Theorem~\ref{thm:bog} instead of \ref{thm:bog-exp}, we obtain an analog of
the first claim in Theorem~\ref{thm:bog-exp} for general finite abelian
groups. We are not sure if some variant of this result can be used to
prove a removal lemma.

\begin{proposition}
For every $\epsilon > 0$ and $D = D(\epsilon) \to \infty$ as $\epsilon
\to 0$, if $G$ is a finite abelian group, and $A \subseteq G$ has VC
dimension at most $d$, then there exist some proper coset
progression $P$ of dimension at most $D$ and size $|P| \ge \epsilon^{d +
o(1)} |G|$, such that $|(A + x) \Delta A|\le \epsilon |G|$ for all $x
\in P$.  Here $o(1)$ is some quantity that goes to zero as $\epsilon
\to 0$, at a rate depending on $d$ and $D$.
\end{proposition}

We conclude with the following related question that we do not know
how to answer (even for $k=2$).  An affirmative answer would strengthen
Szemer\'edi's theorem.

\begin{question}
	Let $k$ be a positive integer and $\delta > 0$. Let $p$ be a sufficiently large prime, and $A \subseteq \ZZ/p\ZZ$ with $\delta p \le \abs{A} \le (1-\delta)p$. Can we always find a $2k$-term arithmetic progression in $\ZZ/p\ZZ$ where the first $k$ terms lie in $A$ and the last $k$ terms lie outside of $A$?
\end{question}

If $p$ had a small prime factor, then taking $A$ to be a non-trivial
subgroup of $\ZZ/p\ZZ$ gives a counterexample. To see the relevance to the
rest of this paper, observe that such a $2k$-term arithmetic progression
would bi-induce a half-graph on $2k$ vertices. For example, if $x -
(k-1)d, x-(k-2)d, \dots, x \in A$ and $x+d,  \dots, x + kd \notin A$,
then $x_i = x - id$ and $y_j = jd$ have the property that, for $1 \le i,
j\le d$, $x_i + y_j \in A$ if and only if $j \le i$.

%
%


\begin{dajauthors}
\begin{authorinfo}[alon]
  Noga Alon\\
  Schools of Mathematics and Computer Science, Tel Aviv University, Tel Aviv 69978, Israel \\
  and \\
  Department of Mathematics, Princeton University, Princeton, NJ 08544, USA \\
  nogaa\imageat{}tau\imagedot{}ac\imagedot{}il\\
  \url{https://web.math.princeton.edu/~nalon/}
\end{authorinfo}
\begin{authorinfo}[fox]
  Jacob Fox\\
  Department of Mathematics, Stanford University,
Stanford, CA 94305, USA \\
  jacobfox\imageat{}stanford\imagedot{}edu \\
  \url{http://stanford.edu/~jacobfox/}
\end{authorinfo}
\begin{authorinfo}[zhao]
  Yufei Zhao\\
  Department of Mathematics, Massachusetts Institute of Technology, Cambridge, MA 02139, USA \\
  yufeiz\imageat{}mit\imagedot{}edu\\
  \url{http://yufeizhao.com}
\end{authorinfo}
\end{dajauthors}


\begin{thebibliography}{99}

\bibitem{ABBM} N. Alon, J. Balogh, B. Bollob\'as, and R. Morris,
The structure of almost all graphs in a hereditary property, {\it
J. Combin. Theory Ser. B} {\bf 101} (2011), 85--110.

\bibitem{AFKS}
N. Alon, E. Fischer, M. Krivelevich and M. Szegedy,
Efficient testing of large graphs, 
{\it Combinatorica} 20 (2000), 451-476.

\bibitem{AFN} N. Alon, E. Fischer, and I. Newman, Efficient testing of
bipartite graphs for forbidden induced subgraphs, {\it SIAM J. Comput.}
{\bf 37} (2007), 959--976.

\bibitem{BGT1} E. Breuillard, B. Green, and T. Tao, Small doubling in
groups. {\it Erd\H{o}s centennial}, 129--151, Bolyai Soc. Math. Stud.,
25, J\'anos Bolyai Math. Soc., Budapest, 2013.

\bibitem{BGT2} E. Breuillard, B. Green, and T. Tao, The structure of
approximate groups, {\it Publ. Math. Inst. Hautes \'Etudes Sci.} {\bf 116}
(2012), 115--221.

\bibitem{CPT} G. Conant, A. Pillay and C. Terry, A group version of
stable regularity, {\it Math. Proc. Camb. Philos. Soc.}, to appear.

\bibitem{CF12} D. Conlon and J. Fox, Bounds for graph regularity and
removal lemmas, {\it Geom. Funct. Anal.} {\bf 22} (2012), 1191--1256.

\bibitem{CF13} D. Conlon and J. Fox, Graph removal lemmas. {\it Surveys
in combinatorics} {\bf 2013}, 1--49, London Math. Soc. Lecture Note Ser.,
409, Cambridge Univ. Press, 2013.


\bibitem{Fox11} J. Fox, A new proof of the graph removal lemma, {\it
Ann. of Math.} 174 (2011), 561--579.

\bibitem{FGSY} J. Fox, L. Gishboliner, A. Shapira, and R. Yuster, The
removal lemma for tournaments, {\it J. Combin. Theory Ser. B}, to appear.

\bibitem{FPS} J. Fox, J. Pach, and A. Suk, Erd\H{o}s-Hajnal conjecture
for graphs with bounded VC-dimension, to appear in {\it Discrete
Comput. Geom.}, SoCG 2017 Special Issue.

\bibitem{Gow97} W. T. Gowers, Lower bounds of tower type for Szemer\'edi's uniformity lemma, \textit{Geom. Funct. Anal.} \textbf{2} 1997, 322--337.

\bibitem{Green05} B. Green, A Szemer\'edi-type regularity lemma in
abelian groups, with applications, \textit{Geom. Funct. Anal.} \textbf{15}
(2005), 340--376.

\bibitem{GR} B. J. Green and I. Z. Ruzsa, Freiman's theorem in an
arbitrary abelian group, {\it J. London Math. Soc. (2)} {\bf 75}
(2007), 163--175.

\bibitem{H95} D. Haussler, Sphere packing numbers for subsets of
the Boolean $n$-cube with bounded Vapnik-Chervonenkis dimension,
{\it J. Combin. Theory Ser. A}, {\bf 69} (1995), 217--232.

\bibitem{Kn} M. Kneser, Absch\"atzungen der asymptotischen Dichte von
Summenmengen, {\it Math. Zeitschr.} (in German) {\bf 58} (1953), 459--484.

\bibitem{KSV09} D. Kr\'al', O. Serra and L. Vena, A combinatorial proof of the removal lemma for groups, {\it J.
Combin. Theory Ser. A} {\bf 116} (2009), 971--978.

\bibitem{KSV12} D. Kr\'al', O. Serra and L. Vena, A removal lemma for systems of linear equations over finite fields, {\it Israel J. Math.} {\bf 187} (2012), 193--207.

\bibitem{LS} L. Lov\'asz and B. Szegedy, Regularity partitions and
the topology of graphons, An Irregular Mind, Imre B\'ar\'any, J\'ozsef
Solymosi, and G\'abor S\'agi editors, \emph{Bolyai Society Mathematical
Studies} \textbf{21} (2010), 415--446.

\bibitem{MaSh} M. Malliaris and S. Shelah, Regularity lemmas for stable
graphs, {\it Trans. Amer. Math. Soc.} {\bf 366} (2014), 1551--1585.

\bibitem{Ruz94} I. Z. Ruzsa, Generalized arithmetical progressions and
sumsets, {\it Acta Math. Hungar.} {\bf 65} (1994), 379--388.

\bibitem{Ruz99} I. Z. Ruzsa, An analog of Freiman's theorem in groups,
{\it Ast\'erisque} {\bf 258} (1999), 323--326.

\bibitem{RS} I. Z. Ruzsa and E. Szemer\'edi, Triple systems with no six
points carrying three triangles, in Combinatorics (Keszthely, 1976),
Coll. Math. Soc. J. Bolyai 18, Volume II, 939--945.

\bibitem{San12} T. Sanders, On the Bogolyubov-Ruzsa lemma, {\it  
Anal. PDE} {\bf 5} (2012), no. 3, 627--655. 

\bibitem{San13} T. Sanders, The structure theory of set addition revisited,
{\it Bull. Amer. Math. Soc.} {\bf 50} (2013), 93--127.

\bibitem{Sau} N. Sauer, On the density of families of sets, {\it
J. Combinatorial Theory Ser. A} {\bf 13} (1972), 145--147.

\bibitem{Sha10} A. Shapira, A proof of Green's conjecture regarding the removal properties of sets of linear
equations, {\it J. London Math. Soc.} {\bf 81} (2010), 355--373.

\bibitem{She} S. Shelah, A combinatorial problem; stability and order
for models and theories in infinitary languages, {\it Pacific J. Math.}
{\bf 41} (1972), 247--261.

\bibitem{TV} T. C. Tao and H. V. Vu., {\it Additive combinatorics}, 
Cambridge University Press, 2006.

\bibitem{TeWo} C. Terry and J. Wolf, Stable arithmetic regularity in
the finite-field model, {\it Bull. Lond. Math. Soc.} {\bf 51} (2019), 70--88.

\bibitem{Sz76} E. Szemer\'edi, Regular partitions of graphs, Probl\'emes
combinatoires et th\;eorie des graphes (Colloq. Internat. CNRS, Univ.
Orsay, Orsay, 1976), Colloq. Internat. CNRS, vol. 260, CNRS, Paris,
1978, pp. 399--401.

\bibitem{VC} V. N. Vapnik and A. Ja. \v Cervonenkis, The uniform
convergence of frequencies of the appearance of events to their
probabilities (Russian), {\it Teor. Verojatnost. i Primenen.} {\bf 16}
1971, 264--279.

\end{thebibliography}
\end{document}